\documentclass[11pt]{amsart}
\usepackage{amsfonts,amssymb,verbatim,amsmath,amsthm,latexsym,textcomp,amscd}
\usepackage{latexsym,amsfonts,amssymb,epsfig,verbatim}
\usepackage{amsmath,amsthm,amssymb,latexsym,graphics,textcomp, hyperref}
\usepackage{enumerate}
\usepackage[capitalize]{cleveref}
\usepackage{comment}
\usepackage[utf8]{inputenc}
\usepackage{xcolor}
\usepackage{graphicx, psfrag}
\usepackage{eucal}
\makeindex

\title[broken windows only]{Thurston's broken windows only theorem revisited}
\author{Ken'ichi Ohshika}
\address{Department of Mathematics, Gakushuin University, Mejiro, Toshima-ku, Tokyo, Japan}
\email{ohshika@math.gakushuin.ac.jp}

\newcommand{\reals}{\mathbb{R}}

\newcommand{\integers}{\mathbb{Z}}
\newcommand{\naturals}{\mathbb{N}}

\newcommand{\hyperbolic}{\mathbb{H}^3}

\newcommand{\isom}{\mathrm{Isom}}

\newcommand{\ie}{i.e.\ }

\newcommand{\Fr}{\operatorname{Fr}}

\newcommand{\len}{\mathrm{length}}
\newcommand{\trlen}{\mathrm{trlength}}
\newcommand{\wind}{\mathrm{window}}
\newcommand{\tr}{\mathrm{tr}}
\renewcommand{\int}{\mathrm{Int}}

\newcommand{\teich}{\mathcal{T}}
\let \pslc \PSL

\newtheorem{theorem}{Theorem}[section]
\newtheorem{proposition}[theorem]{Proposition}
\newtheorem{lemma}[theorem]{Lemma}
\newtheorem{corollary}[theorem]{Corollary}

\theoremstyle{definition}

\theoremstyle{remark}

\includecomment{comments}
\specialcomment{comments}{\color{red}}{\color{black}}

\begin{document}
\begin{abstract}
The \lq broken windows only theorem' is the main theorem of the third paper among a series of the paper in which Thurston proved his uniformisation theorem for Haken manifolds.
In this chapter, we show that the second statement of this theorem is not valid, giving a counter-example.
We also give a weaker version of this statement with a proof.
In the last section, we speculate on how this second statement was intended to constitute a proof of the bounded image theorem, which constituted a key of the uniformisation theorem.
The proof of the bounded image theorem was  obtained only quite recently, although its weaker version, which is sufficient for the proof of the uniformisation theorem, had already been proved.
\end{abstract}
\maketitle
\sloppy
\setcounter{section}{-1}
\section{Introduction}
One of the most celebrated achievements of Thurston is the uniformisation of Haken manifolds, which motivated the general geometrisation theorem,  proved later by Perelman.
Thurston planned to publish the proof of the uniformisation theorem in a series of seven papers.
Only the first one was officially published, and the second and the third papers exist as preprints, and are now included in the collected works of Thurston \cite{ThC}.
In the first paper, the compactness theorem for the deformation spaces of acylindrical hyperbolic 3-manifolds is proved.
The third paper generalises this compactness to a general atoroidal 3-manifold with incompressible boundary as a relative convergence theorem, which is dubbed as the \lq broken windows only' theorem (Theorem 0.1 in \cite{Th3}).
This should have constituted the key part of the proof of what is called the bounded image theorem, one of the main ingredients of the entire proof.
The theorem consists of two statements and we shall show that in fact the second statement is not valid, giving a counter-example.

Although Thurston did not write a proof of the bounded image theorem,  a weaker version of the bounded image theorem, the bounded orbit theorem, which first appeared in \cite{ThO}, is sufficient for the proof of the uniformisation theorem as was observed also by Morgan \cite{Mo}, and a detailed proof of this weaker version was given in Kapovich \cite{Kap}.
A proof of the bounded image theorem by Thurston had been unknown, except for the case of acylindrical case given by Kent \cite{Kent} and Brock--Bromberg-Canary--Minsky \cite{BBCM}, but quite recently a proof relying on modern techniques of Kleinian group theorem (in particular the work of Brock--Bromberg--Canary--Lecuire \cite{BBCL} among other works) was  given by Lecuire--Ohshika \cite{LO}.
It is still worthwhile to consider how Thurston himself tried to prove the bounded image theorem using the tools available at that time.
We shall show how this should have been done provided that the broken windows only theorem were true, in the last section.

We now state this broken windows only theorem and explain what is problematic.
For that, we first need to explain some terms used in the theorem, whose precise definitions will be given in the following section.
Let $(M,P)$ be a pared manifold, that is,  a compact irreducible 3-manifold $M$ with a union $P$ of incompressible tori and annuli on $\partial M$ such that $\partial M \setminus P$ is incompressible, and every $\pi_1$-injective, immersed torus or properly immersed annulus can be properly homotoped into $P$.
We denote by $AH(M,P)$ the set of complete hyperbolic metrics on $\int M$ modulo isotopy such that each component of $P$ corresponds to a parabolic end.
The set $AH(M,P)$ corresponds one-to-one to the set of faithful discrete representations of $\pi_1(M)$ into $\pslc$ sending $\pi_1(P_0)$ to parabolic a parabolic subgroup for each component $P_0$ of $P$, modulo conjugacy.

We consider the characteristic submanifold $\Phi$ of $(M,P)$ in the sense of Jaco--Shalen--Johannson, consisting of $I$-bundles (over incompressible surfaces which are neither disc nor annuli nor M\"{o}bius bands), solid tori intersecting $\partial M$ along incompressible annuli on their boundaries, and thickened torus intersecting $\partial M$ along a torus and annuli.
The window $F$ of $(M,P)$ is an $I$-bundle constructed from the characteristic submanifold as follows.
All $I$-bundle components of $\Phi$ are contained in $F$.
For each solid torus or thickened torus component $V$ of $\Phi$, we consider each component of $\Fr V$, which is an annulus.
If it is not properly isotopic into the $I$-bundle component we consider its thin regular neighbourhood, which is an $I$-bundle over an annulus, and make it be included in $F$.
Otherwise we abandon that component.
If there are more than two properly homotopic frontier components, we do this procedure only for one among them.
Thus obtained $F$, which is an $I$-bundles over possibly disconnected surface, is called the window of $(M,P)$.
The window of $(M,P)$ is denoted by $\wind(M, P)$ and its base space by $ \mathrm{wb}(M, P)$.

Now  Thurston's \lq broken windows only' theorem reads as follows:

\begin{theorem}[Thurston \cite{Th3}]
\label{Thurston}
If $\Gamma \in \pi_1(M)$ is any subgroup which is conjugate to the fundamental group
of a component of $M \setminus\wind(M, P )$, then the set of representations of $\Gamma$ in $\mathrm{Isom}(\hyperbolic)$
induced from $AH(M, P)$ are bounded, up to conjugacy.

Given any sequence $N_i \in AH(M, P)$, there is a subsurface with incompressible boundary
$x \subset \mathrm{wb}(M, P)$ and a subsequence $N_{i(j)}$ such that the restriction of the associated sequence
of representations $\rho_{i(j)} : \pi_1(M) \to \isom(\hyperbolic)$ to a subgroup $\Gamma \subset \pi_1(M)$ converges if and only
if $\Gamma$ is conjugate to the fundamental group of a component of $M\setminus X$, where $X$ is the total
space of the interval bundle above $x$. Furthermore, no subsequence of $\rho_{i(j)}$ converges on any
larger subgroup.
\end{theorem}

Thurston gave a fairly detailed proof for the first sentence.
There is also an alternative approach for the first part using Morgan--Shalen theory \cite{MS1, MS2, MS3}, and also a stronger version of this first part has been proved by Brock--Bromberg--Canary--Minsky \cite{BBCM}.
We shall give a counter-example to the second sentence of \cref{Thurston}.
The problem lies in the fact that there may be a solid torus $V$ in the characteristic submanifold of $M$ which intersects $\partial M$ at more than three annuli. Such a solid torus may bring about a $3$-manifold  homotopy equivalent to $M$ which is not homeomorphic to $M$, by changing the order of attaching the components of $M \setminus V$ along $\partial V$.
We call this operation {\em shuffling} along $V$.

We shall show that even though \cref{Thurston} does not hold in the original form, a weaker version as follows can be proved making use of the Morgan--Shalen theory and the efficiency of pleated surfaces due to Thurston \cite{Th2}.

\begin{theorem}
\label{rectified Thurston}
Let $(M,P)$ and $N_i$ be as given in \cref{Thurston}.
Then there is a pared manifold $(M', P')$ obtained from $(M,P)$ by shuffling around solid pairs and thickened torus pairs of the characteristic submanifold of $(M,P)$ for which the following hold.

Let $W'$ be the characteristic submanifold of $(M',P')$.
Let $\mathcal W$ be the union of its $I$-pairs, and $\mathcal V$ the union of solid torus pairs and thickened torus pairs of $W'$.
Let $p\colon \mathcal W \to w$ be the fibration as an $I$-bundle for $\mathcal W$.
Then there is an  incompressible subsurface $x$ of $w$ and a union of essential annuli (possibly with singular axes) $\mathcal A$  in $\mathcal V$ such that the restriction of the associated sequence
of representations $\rho_{i(j)} : \pi_1(M) \to \isom(\hyperbolic)$ to a subgroup $\Gamma \subset \pi_1(M)$ converges if and only
if $\Gamma$ is conjugate to the fundamental group of a component of $M\setminus (X \cup \mathcal A)$ under the identification of $\pi_1(M)$ with $\pi_1(M')$, where $X=p^{-1}(x)$.
(When we cut $M$ along an annulus with singular axis, we regard a regular neighbourhood of the singular axis as contained in both sides after the operation.)
Furthermore, no subsequence of $\rho_{i(j)}$ converges on any larger subgroup.
\end{theorem}

Unfortunately,  this weaker version is not sufficient for the proof of the bounded image theorem.

%
%
\section{Preliminaries}
In this section, we explain notions and terminologies which are necessary for  \cref{Thurston,rectified Thurston} and their proofs.
\subsection{JSJ decompositon}
In this subsection, we shall review the theory of Jaco--Shalen and Johannson concerning characteristic decomposition of compact irreducible 3-manifolds by tori and annuli and homotopy equivalences between  such 3-manifolds.

Unless otherwise mentioned, throughout this section $M$ will denote a compact irreducible 3-manifold with (possibly empty) incompressible boundary.
We always assume 3-manifolds to be orientable.
Jaco--Shalen and Johannson proved that  there exists a disjoint family of embedded incompressible tori $T_1, \dots, T_p$ and properly embedded incompressible annuli $A_1, \dots, A_q$ with the following properties.
\begin{enumerate}
\item
Cutting $M$ along $(T_1 \cup \dots \cup T_p\cup A_1 \cup \dots \cup A_q)$, we obtain compact 3-submanifolds of $M$ each of which is one of the following.
\begin{enumerate}
\item
A Seifert fibred manifold whose frontier in $M$ consists of fibred tori and fibred annuli.
\item
An atoroidal manifold, \ie a manifold such that every $\pi_1$-injective map from a torus to the manifold is homotoped into a boundary component.
We call each of such manifolds a {\em hyperbolic piece}.
(In particular the definition implies that each of the tori $T_1, \dots , T_p$ is disjoint from the interior of any hyperbolic piece.)
\end{enumerate}
\item If a hyperbolic piece $N$ contains some of $A_1, \dots, A_q$, then the 3-manifold $N'$ obtained by cutting $N$ along the annuli contained in $N$ consists of three types of manifolds.
\begin{enumerate}
\item An $I$-bundle over a surface such that its intersection with $N$ coincides with the associated $\partial I$-bundle and is an incompressible subsurface of $\partial N$.
Such a component is called an {\em $I$-pair}\index{$I$-pair}, and we denote each of them in the form of a pair $(\Phi, \Sigma)$, where $\Phi$ is an $I$-bundle and $\Sigma = N \cap \partial M$, or simply by $\Phi$ if there is no fear of confusion.
\item
A solid torus $V$ whose intersection with $\partial N$ consists of annuli which are incompressible both on $\partial V$ and $\partial N$.
Such a component $(V, V \cap \partial N)$ is called a {\em solid torus} pair\index{solid torus pair}.
\item A thickened torus $U\cong S^1 \times S^1 \times I$ such that $S^1 \times S^1 \times \{0\}$ is contained in $\partial N$ and $S^1 \times S^1 \times \{1\}$ intersects $\partial N$ along a union of incompressible annuli no two of which are homotopic in $\partial N$.
For such a component $(U, U \cap \partial N)$ is called a {\em thickened torus pair}\index{thickened torus pair}.
\item An acylindrical manifold, \ie a compact 3-manifold such that every properly embedded annulus is relatively properly homotopic into a boundary component.
\end{enumerate}
%
\end{enumerate}

We note that every $\pi_1$-injective proper map from an annulus into a hyperbolic piece is properly homotopic into either an $I$-pair or a solid torus pair or a thickened torus pair.
We take a family of incompressible tori and annuli having the above properties which is minimal with respect to the inclusion among the families with the same properties.
We call a decomposition of $M$ by such a minimal family a {\em JSJ-decomposition}\index{JSJ-decomposition}.
Jaco--Shalen \cite{JS} and Johannson \cite{Jo} proved that JSJ-decomposition is unique up to ambient isotopy.

We now assume that $M$ is atoroidal.
In this case, the union of the $I$-pairs, the solid torus pairs and the thickened torus pairs is called the {\em characteristic submanifold}\index{characteristic submanifold} of $M$.
The {\em window}\index{window} $F$ of $M$ is constructed from the characteristic submanifold as follows.
We first consider the $I$-pairs in the characteristic submanifold, and let them be included in $F$.
For a solid torus pair or a thickened torus pair $V$,  we replace it with  regular neighbourhoods $V'_1, \dots , V'_k$ of the components of $\Fr V$, which are  solid tori and can be regarded as $I$-bundles over annuli.
If $V_j'$ is properly homotopic into an $I$-pair or another solid torus already added to $F$, we discard it, and otherwise we let $V_j'$ be included in $F$.

\subsection{Homotopy equivalences and books of $I$-bundles}
Each $I$-pair, solid torus pair and thickened torus pair in the characteristic submanifold of an atoroidal Haken manifold may give rise to a homotopy equivalence which is not realised by a homeomorphism under some conditions as will be explained below.
Let $(\Phi , \Sigma)$ be an $I$-pair in a JSJ-decomposition of an atoroidal Haken manifold $M$.
By definition, $\Phi$ has an $I$-bundle structure whose projection to its base surface we denote by $p\colon \Phi \rightarrow B$.
Suppose that $\partial B$ has more than one component, and let $b$ be a component of $\partial B$.
Its inverse image $p^{-1}(b)$ is an annulus which is a union of fibres.
We parametrise $p^{-1}(b)$ as $\{(z, t)\mid z=e^{ i \theta}, \theta \in [0,2\pi), t \in [0,1]\}$ so that the second coordinate corresponds to fibres.
We consider an involution $\iota$ on $p^{-1}(b)$ taking $(z,t)$ to $(z, 1-t)$.
Now we remove $N$ from $M$ and paste it back to $M\setminus N$ by reversing the direction of each fibre in $p^{-1}(b)$ by  the involution $\iota$ just defined.
Then we get another 3-manifold $M'$ homotopy equivalent to $M$, and this operation gives a homotopy equivalence from $M$ to $M'$.
We call this operation  {\em flipping}\index{flipping} on the $I$-pair $(\Phi, \Sigma)$, following Canary-McCullough \cite{CM}, where homotopy equivalences of pared manifolds are thoroughly studied.
We note that $M'$ is non-orientable if the annulus $p^{-1}(b)$ is non-separating.
The same kind of operation can be defined on a frontier component of a solid torus pair or a thickened torus pair in the characteristic submanifold, which we also call  flipping.

Next let $(V, C)$ be a solid torus pair in the JSJ-decomposition of $M$.
We call the number of components of $V \cap \partial M$ the {\em valency} of $V$.
It is evident that the valency is also equal to the number of the components of $\Fr V$.
Suppose that the valency of $V$ is greater than $2$.
Then we can remove $V$ from $M$ and paste it back after changing the cyclic order of the annuli in $\Fr V$.
In this way, we get a 3-manifold homotopy equivalent to $M$.
We call this operation {\em shuffling}\index{shuffling} around $V$.
We can define shuffling around a thickened torus pair $(U, U \cap \partial M)$ in the same way, to be changing the order of  pasting the components of $M \setminus U$ to $\Fr U$.

Jaco--Shalen \cite{JS} and Johannson \cite{Jo} proved that any homotopy equivalence between two compact irreducible 3-manifolds with (non-empty) incompressible boundaries is obtained by repeating flipping and shuffling.

We now consider a special case of 3-manifolds obtained by pasting more than two product $I$-bundles along one solid torus.
Let $V$ be a solid torus and take a simple closed curve $\alpha$ on $\partial V$ homotopic to the $p$-time iteration of a core curve with $p \geq 1$.
Letting $n$ be an integer greater than $2$, we consider $n$ compact surfaces $\Sigma_1, \dots , \Sigma_n$ each of which has only one boundary component, denoted by $c_1 \dots c_n$.
We then consider product $I$-bundles $\Sigma_1 \times I, \dots , \Sigma_n \times I$, and paste them along $c_1 \times I , \dots , c_n \times I$ to $\partial V$ so that all of these annuli are identified with disjoint annuli with core curves isotopic to $\alpha$ which lie on $\partial V$ in the (cyclic) order of the subscripts.
We get a compact, irreducible, atoroidal 3-manifold $N$ in this way, and it is called a {\em book of (product) $I$-bundles}\index{book of $I$-bundles} with $n$ pages.
Such manifolds were studied in the context of deformation spaces of Kleinian groups for the first time by Anderson-Canary \cite{AC}.
In particular when $p=1$,  the book of $I$-bundles is said to be {\em primitive}.
The characteristic submanifold of $N$ consists of $n+1$ components, \ie $I$-pairs corresponding to $\Sigma_1 \times I, \dots , \Sigma_n \times I$ and a solid torus pair $(V, V \cap \partial N)$.
The window consists of $n$ components: the characteristic submanifold with the solid torus pair $(V, V \cap \partial N)$ removed.
By the theory of Jaco-Shalen and Johannson mentioned above, any homotopy equivalence from $N$ to another compact irreducible 3-manifold is obtained by shuffling along $V$ since no flipping gives rise to a new manifold.

\subsection{Pared manifolds}
A {\em pared manifold}\index{pared manifold} is a pair $(M,P)$ of a compact irreducible 3-manifold $M$ (with possibly compressible boundary) and a disjoint union of incompressible tori and annuli $P$ lying on $\partial M$ with the following properties.
\begin{enumerate}[(i)]
\item Any $\pi_1$-injective immersion from a torus is homotopic into $P$.
\item Any $\pi_1$-injective proper  immersion of an annulus whose boundaried is mapped into $P$ is properly homotopic into $P$.
\end{enumerate}
We also assume the following:
\begin{enumerate}[(i)]
\setcounter{enumi}{2}
\item Every component of $\partial M \setminus P$ is incompressible.
\end{enumerate}

We define the characteristic submanifold of  a pared manifold $(M,P)$ in the same way as before, just imposing the condition that no component of  $P$ is contained in the associated $\partial I$-bundle of an $I$-pair.

\subsection{Deformation spaces}
Let $M$ be a compact, irreducible, atoroidal $3$-manifold.
We say that $M$ admits a convex compact hyperbolic structure when there is a convex compact hyperbolic 3-manifold $N$ and a homeomorphism from $M$ to $N$ is given.
The hyperbolic manifofold $N$ can be embedded in a complete hyperbolic manifold homeomorphic to $\int N$ as a convex core.
Therefore a convex compact hyperbolic structure on $M$ also induces a complete hyperbolic structure on $\int M$.

Each complete hyperbolic structure on $\int M$ corresponds to a faithful discrete representation of $\pi_1(M)$ into $\pslc$, which is uniquely determined up to conjugacy.
The set of faithful discrete representations of $\pi_1(M)$ into $\pslc$ modulo conjugacy is denoted by $AH(M)$\index{$AH(M)$} and is given a topology induced from the topology of element-wise convergence of the representation space.
We consider the space of convex compact hyperbolic structures on $M$ modulo isotopy, which we denote by $QH_0(M)$, and is regarded as a subspace of $AH(M)$ by using extensions of such structures to complete hyperbolic structures as mentioned above.
By the theory of Ahlfors, Bers, Kra, Maskit, Marden and Sullivan, there is a covering map $q \colon \teich(\partial M) \to QH_0(M)$.
(The definition of the map in this general setting is due to Bers \cite{Bers2}, Kra \cite{Kra}, Maskit \cite{Mas} and Marden \cite{Mar}.
Sullivan \cite{Su} showed that the map is surjective.) 
Under the present assumption that  $\partial M$ is incompressible $q$ is known to be a homeomorphism (\cite{Bers2,Kra,Mas,Mar}).
In the special case when $M$ is homeomorphic to $S \times I$ for a closed surface, this parametrisation is what is called the Bers simultaneous uniformisation \cite{Ber}, and is expressed as $qf \colon \teich(S) \times \teich(\bar S) \to QH_0(M)$, where $\teich(S)$ denotes the Teichm\"{u}ller space of oriented $S$, whereas $\teich(\bar S)$ denotes that of $S$ with orientation reversed.
Representations contained in $QH_0(M)$ in this case are called quasi-Fuchsian representations.

For a pared manifold $(M,P)$, we define $AH(M,P)$ to be the set of faithful discrete representations of $\pi_1(M)$ into $\pslc$ which send  each non-trivial element of $\pi_1(P_0)$ to a parabolic element for each component $P_0$ of $P$.
The set modulo isotopy of convex finite-volume complete hyperbolic structures on $(M,P)$ for which $P$ corresponds to parabolic ends forms a subspace of $AH(M,P)$, which we denote by $QH_0(M,P)$.
As in the case of convex compact hyperbolic structures, there is a homeomorphism $q \colon \teich(\partial M \setminus P) \to QH_0(M)$.

\subsection{Morgan--Shalen compactification}
In this subsection we summarise some results contained in Morgan-Shalen \cite{MS1, MS2, MS3}, which will be used in the proof of \cref{rectified Thurston} in \cref{rectification}.

Let $M$ be an irreducible, atoroidal compact 3-manifold with incompressible boundary.
We may also consider a pared manifold $(M,P)$ as will be explained later.
What interests us is the case when $M$ has a boundary component which is not a torus.
Let $\{\rho_i\}$ be a sequence in $AH(M)$.
Suppose that $\{\rho_i\}$ does not have a convergent subsequence in $AH(M)$.
Fix a generator system $g_1, \dots , g_k$ of $\pi_1(M)$ such that the trace function $\tr_{g_1}, \dots , \tr_{g_k}$ generates the trace functions of elements of $\pi_1(M)$ as a ring.
(See Culler--Shaen \cite{CS} for the existence of such a system.)
Let $\mu_i$ be $\max_{j=1}^k\{\trlen(\rho_i(g_j))\}$, where $\trlen$ denotes the translation length.
By changing the ordering of the generator system and passing to a subsequence, we can assume that $\mu_i$ is attained by $\trlen(\rho_i(g_1))$.
Put a basepoint $x_i$ in $\hyperbolic$ which is within bounded distance from the axis of $\rho_i(g_1)$.
Consider the rescaled hyperbolic space $\trlen(\rho_i(g_1))^{-1} \hyperbolic$ with basepoint at $x_i$, and take an equivariant Gromov limit.
Then we get an isometric action of $\pi_1(M)$ on an $\reals$-tree.
(Morgan--Shalen \cite{MS1}, Paulin \cite{Pa}, Bestvina \cite{Be})

Here an {\em $\reals$-tree}\index{$\reals$-tree} is a geodesic space in which for any two points, there is only one simple arc connecting them.
In other words, it is a $0$-hyperbolic space in sense of Gromov.
An isometric action of $G$ on an $\reals$-tree $T$ is said to have {\em small edge stabilisers} when for any non-trivial arc $a$ in $T$, its stabiliser is a virtually abelian group.
It is said to be {\em minimal} when $T$ has no $\rho_G(G)$-invariant proper subtree.

In the case when we consider a pared manifold $(M,P)$, we assume that for any component $P_0$ of $P$ and any element $\gamma \in \pi_1(P_0)$, the translation length of $\rho_i(\gamma)$ is bounded as $i \to \infty$.
Then the limit action on an $\reals$-tree has a property that for any element $\gamma$ in $\pi_1(P_0)$ as above, the action of $\gamma$ on the tree is trivial.

\begin{theorem}[Morgan--Shalen, Paulin, Bestvina]
\label{Morgan-Shalen 1}
As a rescaled Gromov limit, we obtain a minimal isometric action of $\pi(M)$ on an $\reals$-tree $T$ with  small edge stabilisers.
\end{theorem}

A codimension-1 lamination in a compact $3$-manifold $M$ is a closed subset $L$ of $M$ such that for any point $x \in M$ there is a coordinate neighbourhood $U$ of $M$ in the form $D^2 \times [0,1]$, where $U \cap L=D^2 \times J$ with a closed subset $J$ of $[0,1]$.
We can define a leaf of $L$ in the same way as for a foliation.
A codimension-1 lamination is said to be measured when there is a positive Borel measure on $[0,1]$ supported on $J$ which is invariant under coordinate changes.

We equip the set of codimension-1 measured laminations with a topology induced from the weak topology on the set of measures on transverse arcs.
A positively weighted disjoint union of embedded surfaces is an example of a codimension-1 measured lamination.
A codimension-1 measured lamination is said to be incompressible when it is a limit of positively weighted unions of properly embedded incompressible surfaces.
An incompressible codimension-1 measured lamination is said to be {\em annular} when it is a limit of weighted properly embedded incompressible annuli.

\begin{proposition}[Morgan--Shalen \cite{MS2}]
\label{annular}
Any codimension-1 annular lamination $L$ is properly isotopic into the characteristic submanifold $W$ of $M$.
For each $I$-pair $W_0$ of $W$, the intersection $L \cap W_0$ can be isotoped so that it becomes vertical with respect to the fibration.
\end{proposition}

Let $L$ be an incompressible codimension-1 measured lamination, and consider the universal cover $\tilde M$ of $M$.
(In our setting, $\tilde M$ is always simply connected.)
Then each leaf of $L$ can be lifted to a topological plane in $\tilde M$, and the union of lifted leaves constitute a codimension-1 measured lamination in $\tilde M$.

For an incompressible measured lamination $L$, we can construct an $\reals$-tree {\em dual to $L$} in the following way.
In the universal cover $\tilde M$, we collapse each component of $\tilde M \setminus \tilde L$ to a point,  replace isolated leaves with arcs, and put a metric coming from the lift of the transverse measure of $L$.
Then we get an $\reals$-tree on which $\pi_1(M)$ acts by isometries.
We call this the $\reals$-tree {\em dual to} $L$, and denote it by $T_L$.
We note that we can construct an $\reals$-tree dual to a measured geodesic lamination on a hyperbolic surface in the same way, on which its fundamental group acts by isometries.

A continuous map $f$ from an $\reals$-tree $T$ to another $\reals$-tree $T'$ is said to be a {\em morphism} if for each point $x$ in $T$, there is an geodesic segment $I$ in $T$ containing $x$ such that $f|I$ is an isometry.
An isomorphism from $T$ to $T'$ is a surjective isometry.
A morphism may fail to be isomorphism, for $f$ may be \lq folded' at some points of $T$.

\begin{theorem}[Morgan--Shalen \cite{MS3}, Morgan--Otal \cite{M-O}]
\label{dominating lamination}
Let $\rho \colon \pi_1(M) \to \isom(T)$ be a minimal isometric action on an $\reals$-tree $T$ with small edge stabilisers.
Then there is an annular codimension-1 lamination in $M$ with an $\pi_1(M)$-equivariant morphism $t \colon T_L \to T$, where $T_L$ is the $\reals$-tree dual to $L$.

Furthermore, let $(M,P)$ be  a pared manifold, and $\rho \colon \pi_1(M) \to \isom(T)$  as above has a property that for any component $P_0$ of $P$, its fundamental group  $\pi_1(P_0)$ lies in the kernel of $\rho$.
Then the annular codimension-1 lamination $L$ above can be taken to be disjoint from $P$.
\end{theorem}
We note that this theorem only guarantees the existence of morphism, not an isomorphism.
As we shall see in \cref{sec:counter}, there is an example for which we cannot have an annular codimension-1 lamination giving an isomorphism.

We need to generalise slightly the notion of annular codimension-1 lamination for \cref{rectified Thurston,dual annular}.
We define an {\em annulus with a singular axis}\index{annulus with a singular axis} as follows.
Let $V$ be a solid torus pair in a pared manifold $(M,P)$.
Then $V \cap \partial M$ consists of annuli on $\partial V$ whose core curves are homotopic in $V$ to the $p$-time iteration of a core curve of $V$ for some positive integer $p$.
This means that there is a Seifert fibration $\pi \colon V \to D$ for a disc $D$ with a cone point $\mathfrak c$ of order $p$  with respect to which each component of $V \cap \partial M$ is vertical, i.e. it is expressed as $\pi^{-1}(\alpha)$ for a simple arc on $\partial D$.
We denote the union of the simple arcs corresponding to $V \cap \partial M$ by $A$.
An {\em annulus with a singular axis} is defined to be a singular annulus of the form $p^{-1}(\beta_1 \cup\dots \cup \beta_k)$ for the union of simple arcs $\beta_1, \dots , \beta_k$ with $k \geq 2$ in $D$  connecting $A$ to $\mathfrak c$ which intersect each other only at $\mathfrak c$.
It has a singular axis, which is equal to $\pi^{-1}(\mathfrak c)$ around which any non-singular fibre wraps $p$ times.
We call such an annulus weighted if each of $\beta_1, \dots , \beta_k$, and correspondingly its preimage, has a positive weight.
In  \cref{dual annular}, we shall allow such weighted annuli with singular axes in solid torus pairs to be included as components of an annular codimension-1 measured lamination.
%

\subsection{Skora's theorem}
In the previous subsection, we talked about an action on an $\reals$-trees dual to a codimension-1 measured lamination in a 3-manifold.
In this subsection, in the case of (possibly non-orientable) surfaces of hyperbolic type, any minimal  isometric action on an $\reals$-tree with small edge stabilisers is dual to a measured geodesic lamination on the surface.

The following theorem was proved by Skora \cite{Sk}.

\begin{theorem}[Skora]
\label{Skora}
Let $S$ be a closed hyperbolic surface, and suppose that $\rho \colon \pi_1(S) \to \isom(T)$ is a minimal action with small edge stabilisers.
Then there is a measured geodesic lamination $\lambda$ on $S$ with a $\pi_1(S)$-equivariant isomorphism between the $\reals$-tree $T_L$ dual to $L$ and $T$. 
\end{theorem}

We note that even in the case when $S$ has punctures or boundaries, \cref{Skora} holds under the assumption that elements of $\pi_1(S)$ corresponding to loops around single punctures or boundaries act on $T$ trivially.

\section{Counter-example}
\label{sec:counter}
In this section, we shall construct a counter-example to \cref{Thurston}.
We consider a primitive book of product $I$-bundles explained in the previous section consisting of more than three pages.
It is a special case of Thurston's theorem that such a 3-manifold admits convex compact hyperbolic structure in its interior, but even without using Thurston's theorem, we can construct such a structure just by the Klein--Maskit combination theorem (see \cite{Masbook} for various versions of the combination theorem).
To fix symbols, let $M$ be a primitive book of $I$-bundles which is a union of product $I$-bundles $\Sigma_1 \times [0,1], \dots , \Sigma_n \times [0,1]$ and a solid torus $V$, where each $\Sigma_j$ has only one boundary component, and $\partial \Sigma_i \times [0,1]$ is attached to $\partial V$ along an essential annulus $A_i$ whose core curve is homotopic to a core curve of $V$.


We order the $I$-bundles so that the attaching annuli $A_1, \dots, A_n$ lie on $\partial V$ in this (cyclic) order.
We assume that the ordering is compatible with the parametrisation of $[0,1]$, that is, $\partial \Sigma_i \times \{1\}$ and $\Sigma_{i+1} \times \{0\}$ are juxtaposed.
We let $\Sigma_i'$ be the complement of a thin annular neighbourhood of $\partial \Sigma_i$ in  $\Sigma_i$.
The characteristic submanifold $W$ of $M$ consists of a solid torus $V$ and $I$-bundles $\Sigma_1' \times [0,1], \dots , \Sigma'_n \times [0,1]$, and the window consists of $\Sigma_1' \times [0,1], \dots , \Sigma'_n \times [0,1]$.
We call the boundary component of $A_i$ nearer to $A_{i+1}$, \ie $\partial \Sigma_i \times \{1\}$ (we regard $A_{n+1}$ as $A_{1}$) the upper boundary.
For $1 \leq i < j \leq n$, we  denote by $B_{i,j}$ the closure of a component of the complement of $A_i \cup A_j$ in $\partial V$ which contains the upper boundary of $A_i$.
We let $S_{i,j}$ be $\Sigma_i \times \{1\} \cup  B_{i,j} \cup \Sigma_j \times \{0\}$ which is a closed surface embedded in $M$.

We now define the (right) Dehn twist of $M$ along $A_j$.
The right Dehn twist of $M$ along $A_j$ is obtained by cutting $M$ along $A_i$ into $\Sigma_j \times I$ and the rest, and glue them back after twisting  $A_i$ on $\Sigma_j \times I$ by $2\pi$ in the direction of $S^1$  keeping $I$-direction unchanged to the right when we look at it from the inside of $\Sigma_j \times I$. 
In other words, the right Dehn twist along $A_j$ is a product of the right Dehn twist around $\Fr \Sigma_j$ and the identify on the fibre direction.
We denote this homemorphism by $\tau_{A_j} \colon M \to M$.
The left Dehn twist along $A_j$ is defined to be $\tau_{A_j}^{-1}$.

\begin{proposition}
\label{counter}
Let $M$ be a primitive book of  product $I$-bundles of $n$ pages.
Fix a convex cocompact hyperbolic structure $m_0$ on $M$, which is also regarded as a point in $AH(M)$.
Let $\{m_i\in AH(M)\}$ be a sequence obtained from $m_0$ by pushing forward $m_0$ by the composition of the $i$-time iteration of the right Dehn twist $\tau_{A_1}$ along $A_1$ and the $i$-time iteration of the left Dehn twist $\tau_{A_3}^{-1}$ along $A_3$, \ie we define $m_i=(\tau_{A_1} \circ \tau_{A_3}^{-1})_*^{i}(m_0)$.
Then the following hold:
\begin{enumerate}[(a)]
\item
The sequence $\{m_i\}$ does not have a convergent subsequence in $AH(M)$.
\item The sequence of restrictions $\{m_i|\pi_1(S_{1,2})\}$ does not have a convergent subsequence in $AH(S_{1,2} \times I)$.
\item
The restrictions $m_i|\pi_1(S_{1,3})$ converge up to conjugation.
\item
The restrictions $m_i|\pi_1(S_{2,4})$ converge up to conjugation.
\end{enumerate}

\end{proposition}
\begin{proof}
Since (a) follows from (b), we shall first prove (b).
Take a covering $ \tilde M_{1,2} $ of $M$ associated with $\pi_1(S_{1,2})$.
The open 3-manifold $\tilde M_{1,2}$ has a compact core $C_{1,2}$ homeomorphic to $S_{1,2} \times I$ containing homeomorphic lifts of $\Sigma_1 \times I$ and $\Sigma_2 \times I$.
Since convex compact hyperbolic structures are lifted to convex compact hyperbolic structures (see Proposition 7.1 of Morgan \cite{Mo}), 
the sequence $\{m_i\}$ is lifted to a sequence $\{\tilde m_i\}$ in $QH_0(S_{1,2}\times I)\subset AH(\tilde M_{1,2})$.
We shall show that $\{\tilde m_i\}$ does not have a convergent subsequence.

Recall that  $m_i$ is obtained from $m_0$ by the composition of the $i$-time iterations of the right Dehn twists along $A_1$ and $A_3$.
Putting a basepoint $x_0$ on $\Sigma_1\times \{0\}$, we see that every closed loop representing  an element in $\pi_1(S_{1,2} , x_0)$ can be homotoped  off $A_3$ (fixing the basepoint).
On the other hand, the annulus $A_1$ is lifted to an annulus $\tilde A_1$ properly embedded in $C_{1,2}$, and $\tilde m_i$ is obtained from $\tilde m_0$ by the $i$-time iteration of the right Dehn twist along $\tilde A_1$.
Since  $\tilde m_0$ is convex compact,  $\tilde m_0$ corresponds to  a quasi-Fuchsian representation.
By the Bers simultaneous uniformisation, there $\tilde m_0$ is expressed as $qf(g_0, h_0)$, with $(g_0, h_0) \in \teich(S_{1,2}) \times \teich(\bar S_{1,2})$.
Recall that $\hyperbolic/qf(g_0, h_0)$ is homeomorphic to $S_{1,2} \times (0,1)$.
The compact core $C_{1,2}$ is homeomorphic to $S_{1,2} \times [0,1]$, and we can identify
  $S_{1,2}$ and $\bar S_{1,2}$  with $S_{1,2} \times \{0\}$ and $S_{1,2} \times \{1\}$ respectively by homeomorphisms preserving markings, where for the latter the orientation of $\bar S_{1,2}$ is the reverse of the natural one of $S_{1,2} \times \{1\}$.
Let $\gamma, \bar \gamma$ be  simple closed curve on $S_{1,2}$ and $\bar S_{1,2}$ isotopic to $\partial \tilde A_1 \cap S_{1,2} \times \{0\}$ and $\partial A_1 \cap S_{1,2} \times \{1\}$ respectively.
The simple closed curve $\bar \gamma$ corresponds to $\gamma$ on $\bar S$ by the identification above.
By this, $\tilde m_i$ is expressed as $qf(\tau_\gamma^i(g_0), \tau_\gamma^i(h_0))$, where $\tau_\gamma$ denotes the right Dehn twist around $\gamma$.
Then either by considering the geometric limit of the conformal structures at infinity, or just applying  \cite{OhT} or \cite{BBCL}, we see that $\{\tilde m_i\}$ does not have a convergent sequence in $AH(S\times I)$, and hence $\{m_i\}$ does not have a convergent sequence in $AH(M)$.
This completes the proof of (b), and we are also done with (a).

The parts (c) and (d) can be proved by the exact same argument.
Therefore, we shall only prove (c).
As in the case (a), take a covering $\tilde M_{1,3}$ associated with $\pi_1(S_{1,3})$.
Let $\tilde m_i$ be the convex compact hyperbolic structure on $\int \tilde M_{1,3}$ which is  the lift of $m_i$.
In the present case, both of the annuli $A_1$ and $A_3$ are lifted to properly embedded annuli $\tilde A_1$ and $\tilde A_3$ in $\tilde M_{1,3}$.
We can also take a compact core $C_{1,3}$ containing homeomorphic lifts of $\Sigma_1 \times [0,1]$ and $\Sigma_3 \times [0,1]$.
The right Dehn twist along $A_1$ is lifted to the one along $\tilde A_1$, and the left Dehn twist along $A_3$ is lifted to the one along $\tilde A_3$.
Now the annuli $\tilde A_1$ and $\tilde A_3$ are parallel in $C_{1,3}$, and hence the right Dehn twist along $\tilde A_1$ and the left Dehn twist along $\tilde A_3$ cancel out each other.
Therefore we see that $\tilde m_i=\tilde m_0$ for all $i$, which evidently implies that $\{m_i|\pi_1(S_{1,3})\}$ converges up to conjugation.
\end{proof}

Now we show that the sequence in \cref{counter} is a counter-example of the second claim of \cref{Thurston}.

\begin{corollary}
\label{disproof}
The sequence $\{m_i\}$ in \cref{counter} disproves the second sentence of \cref{Thurston}.
\end{corollary}
\begin{proof}
Recall that $M$ is a primitive book of product $I$-bundles, and its window, which we denote by  $F$, is the union of $\Sigma_1' \times I, \dots , \Sigma'_n \times I$.
We denote the fibring of $F$ as an $I$-bundle by $p\colon F \to B$.
Suppose, seeking a contradiction, that the second sentence of \cref{Thurston} holds.
Since $\{m_i\}$ in \cref{counter} diverges, there must be an incompressible subsurface $x$ of the base surface $B$ with the property of \cref{Thurston} that the restriction of $m_i$ to a subgroup of $\pi_1(M)$ converges if and only if $\Gamma$ corresponds to the fundamental group of a component of $M \setminus X$ for $X=p^{-1}(x)$.

Recall that for each component $\Sigma_j' \times I$ of $F$, the restrictions $m_i|\pi_1(\Sigma_j' \times I)$ converge.
Therefore, $\Sigma'_j \times I$ must lie outside $X=p^{-1}(x)$ for every $j$.
This forces $x$ to be a union of annular neighbourhoods of some of the boundaries $\partial \Sigma'_j\ (j=1, \dots , n)$.
If $x$ contains an annular neighbourhood of  $\partial \Sigma'_1$, then there is no component of $M \setminus X$ corresponding the subgroup $\pi_1(S_{1,3})$, contradicting (c) of \cref{counter}.
Therefore, $x$ cannot contain such an annulus.
In the same way, by (c) of \cref{counter}, $x$ cannot contain an annular neighbourhood of $\partial \Sigma'_2$.
Now, these show that $S_{1,2} \times I$ is isotoped into a component of $M \setminus X$.
This contradicts (b) of \cref{counter}.
Thus we are lead to a contradiction.
\end{proof}

We next show that \cref{counter} gives an example of $\reals$-tree action which is not dual to any codimension-1 annular measured lamination.
We note that this also shows that the second sentence of \cref{Thurston} cannot be remedied just by taking into account annuli embedded in solid torus pairs, replacing the window in the statement to the entire characteristic submanifold.
We recall that in two-dimensional setting, by Skora's theorem, every isometric action of $\pi_1(S)$ on an $\reals$-tree with small stabilisers is dual to a measured lamination.

\begin{corollary}
\label{not Skora}
Let $\{m_i\}$ be a sequence in $AH(M)$ given in \cref{counter}.
Let $\rho \colon \pi_1(M) \to \isom(T)$ be a rescaled limit of $\{m_i\}$, which is a minimal action on an $\reals$-tree $T$ with small edge-stabilisers as in \cref{Morgan-Shalen 1}.
Then there is no codimension-1 annular measured lamination $L$ in $M$ with an equivariant isomorphism from the dual tree $T_L$ to $T$.
\end{corollary} 
\begin{proof}
Suppose, seeking a contradiction, that there is such a codimension-1 measured lamination $L$.
By \cref{annular}, $L$ can be assumed to lie in the characteristic submanifold $W$ of $M$.
Since the restrictions of $m_i$ to $\pi_1(\Sigma_j)$ converges for every $j=1, \dots , n$, we see that $L$ must be disjoint from these pages of the book, hence must be contained in the solid torus $V$.
By (c) of \cref{counter},  $S_{1,3} \times I$ can be homotoped so that it becomes disjoint from $L$.
This implies that core curves of the annuli $A_1$ and $A_3$ are homotopic in $M \setminus L$.
In the same way, by (d) of \cref{counter},  core curves of $A_2$ and $A_4$ are homotopic in $M \setminus L$.
In the proof of (d), we only used the fact that $m_i$ is defined  by  the $i$-time iterations of Dehn twists along $A_1$ and $A_3$, and that $S_{2,4}$ can be homotoped off from them.
Therefore, the same argument works  also for $S_{2,j}$ with $j=5, \dots , n$.
This implies that $L \cap V$ must be empty.
Since $\{m_i\}$ diverges, $L$ itself cannot be empty.
This is a contradiction.
\end{proof}
%

\section{A weaker version of Thurston's theorem}
\label{rectification}
As can be seen in the proof of \cref{disproof}, the existence of a solid torus pair in $M$  is the cause of the problem.
On the other hand, such a solid torus gives rise to a homotopy equivalence by shuffling.
In this section, we shall show by moving from $M$ to a homotopy equivalent manifold, and taking into account not only the window but also solid torus pairs and thickened torus pairs of the characteristic submanifold, we can prove a weak version of Thurston's original theorem, stated as \cref{rectified Thurston}. 
Before starting to prove this theorem, we prove the following theorem showing that a minimal isometric action of $\pi_1(M)$ with small edge stabilisers on an $\reals$-tree is isomorphic to an action on an $\reals$-tree dual to an annular measured lamination in a 3-manifold homotopy equivalent to $M$.
This should be contrasted with \cref{not Skora}.

\begin{theorem}
\label{dual annular}
Let $(M,P)$ be a pared manifold as given in \cref{Thurston}.
Let $\rho \colon \pi_1(M) \to \isom(T)$ be a minimal isometric action on an $\reals$-tree $T$ with small edge stabilisers such that every element in $\pi_1(M)$ conjugate into the fundamental group of a component of $P$ is mapped to the identity by $\rho$.
Then there is a codimension-1 annular measured lamination $L'$, with weighted annuli with singular axes allowed, in a pared manifold $(M',P')$ obtained from $(M,P)$ by shuffling, and an equivariant isomorphism from the dual tree $T_{L'}$ to $T$.
\end{theorem}

\begin{proof}
By \cref{dominating lamination}, there is  a  codimension-1 annular measured lamination $L$ in $M$  such that there is a $\pi_1(M)$-equivariant morphism $f$ from the dual tree $T_{L}$ to $T$.
For any component $P_0$ of $P$, an element in $\pi_1(P_0)$ acts trivially on $T$.
Therefore we can choose  $L$  which is disjoint from $P$ as stated in \cref{dominating lamination}.

Let $W_0$ be an $I$-pair in the characteristic submanifold $W$ which contains a component of $L$.
Then $W_0$ is an $I$-bundle over a surface $\Sigma$.
We denote this fibration by $p \colon W_0 \to \Sigma$.
By \cref{annular}, $L_0:=L \cap W_0$ can be assumed to be vertical with respect to $p$.
Let $\tilde W_0$ be a component of the preimage of $W_0$ in $\tilde M$ which is invariant under $\pi_1(W_0)\subset \pi_1(M)$.
Let $T^{W_0}$ be the sub-tree of $T$ corresponding to the rescaled limit of $\tilde W_0$, and $T_L^{W_0}$ the subtree of $T_L$ corresponding to $\tilde W_0$.
Then the restriction $f|T_L^{W_0}$ is a $\pi_1(W_0)$-equivariant morphism.
Since $\pi_1(W_0)$ is isomorphic to $\pi_1(\Sigma)$, and an element of $\pi_1(\Sigma)$ corresponding to a boundary component of $\Sigma$ acts on $T$ trivially, we can apply \cref{Skora}, and see that $T^{W_0}$ is $\pi_1(W_0)$-equivariantly isometric to an $\reals$-tree  $T^{'W_0}$ which is dual to an annular codimension-1 measured lamination $L^{W_0}$ in $W_0$.
We replace $L\cap W_0$ with $L^{W_0}$ and hence $T_L^{W_0}$ with $T^{'W_0}$.
Then $f|T^{W_0}_L$  can be replaced with the equivariant isometry between $T^{'W_0}$ and $T^{W_0}$.
At the same time, every image of $T_L^{W_0}$ under the action of $\pi_1(M)$ and the restriction of $f$ there are modified.
Thus we get an $\reals$-tree $T'$ and a morphism $f' \colon T' \to T$ which is isometry on the orbit of the subtree corresponding to $\tilde W_0$ which is dual to $L^{W_0}$.
Repeating this modification for each $I$-pair of $W$ intersecting $L$, we get an $\reals$-tree $\hat T$ and a morphism $\hat f \colon \hat T \to T$ with an annular codimension-1 measured lamination $\hat L$ to which $\hat f$ is dual.
The morphism $\hat f$ is isometric outside the part corresponding to the preimage of the solid torus pairs and thickened torus pairs in $W$.

Now let $V$ be either a solid torus pair or a thickened torus pair in $W$.
Since $\hat L$ is annular, by isotopying it, we can assume that $V \cap \hat L$ consists of properly embedded incompressible annuli whose boundaries lie in $V \cap (\partial M \setminus P)$, and that no two components of $V \cap \hat L$ are parallel as embeddings in  $(V, V \cap (\partial M \setminus P))$.
In the case when $V$ is a solid torus, we can also assume that it is disjoint from a core curve of $V$.
By isotoping $V \cap \hat L$, we can regard $V$ as a Seifert fibred manifold with base surface $D$, which is a disc with singular point $\mathfrak c$ of order $p$ when $V$ is a solid torus and annulus when $V$ is a thickened torus in such a way that these annuli $V\cap \hat L$ are all vertical with respect to the fibration $\pi \colon V \to D$.
We denote the union of arcs on $\partial D$ which is the image of $p(V \cap (\partial M \setminus P))$ by $A$

The subtree $T_V$ of $\hat T$ corresponding to a component of the preimage of $V$ in $\tilde M$ which is invariant under $\pi_1(V)$ is a simplicial tree each of whose edges is dual to a component of $V \cap \hat L$ and has length equal to the weight given to that component.
We note that  the action of the group $\integers$ corresponding to a regular fibre is trivial.
In the case when $V$ is a solid torus pair, and $p > 1$, the cyclic group $\integers_p$ acts on $T_V$ by isometries, whereas in the case when $V$ is a thickened torus, the infinite cyclic subgroup $\integers$  acts on $T_V$ by isometries.
Both of them are quotients of $\pi_1(V)$ by a normal subgroup generated by a regular fibre.
In the latter case, the stabiliser of the group $\integers$ consists of one point since there is no leaf of $\hat L$ which is a torus.
We modify $T_V$ and $\hat f|T_V$, by removing folds as in the proof of Skora's theorem \cite{Sk}, and make $\hat f|T_V$ an embedding to $T$, whose image we denote by $T'_V$.
We can do this equivariantly with respect to the action of either $\integers_p$ or $\integers$.
Then $T'_V$ is a subtree of $T$ on which either $\integers_p$ or $\integers$ acts by isometries.
We consider the quotient $\bar T_V$ of $T'_V$ by this isometric action of $\integers_p$ or $\integers$, which is a finite tree.

Each of the endpoints of $\bar T_V$, \ie each vertex of valency $1$, corresponds to a component (or components) of $\Fr V$.
Conversely, each component of $\Fr V$ corresponds to a vertex of $\bar T_V$, but not necessarily an endpoint.
We give numbers to the components of $\Fr V$ as $A_1, \dots , A_n$ so that the subscripts coincide with the cyclic order of these components on $\partial V$.
Correspondingly, we number the associated vertices of $\bar T_V$.
Note that it is possible that more than one component of $\partial V$ correspond to the same vertex, and hence the vertex has more than one number.
We also note that in the case when $V$ is a thickened torus pair, there is a vertex $v_P$ in $\bar T_V$ corresponding to the torus component of $V \cap \partial M$ which we denote by $P_0$.
In the case when $V$ is a solid torus and $p>1$, there is a subtree $t_c$ of $\bar T_V$ which is the quotient of the stabiliser of the $\integers_p$ action.


We realise $\bar T_V$ in the base surface $D$ so that the vertices  numbered lie on $\partial D$.
In the case when $V$ is a solid torus pair with $p>1$, we collapse $t_c$ into $\mathfrak c$.
In the case when $V$ is a thickened torus pair, we assume that  $v_P$ lies on $P_0$.
Then it can be regarded as a tree dual to a partition of a disc (possibly with a cone point) or an annululs $D$ by disjoint weighted simple arcs $\alpha$ with endpoints at $A \cup \{\mathfrak c\}$, that is,  each region of the complement corresponds to a vertex, and each arc corresponds to an edge and has weight equal to the length of the edge.
If an edge has a numbered vertex, the dual arc should cut off components of $A$ which are the images of  components of $V \cap (\partial M \setminus P)$ having the numbers given to the vertices.
(See Figure \ref{tree}.)
In the case when $t_c$ contains an edge, each edge is dual to a weighted simple proper arc passing through $\mathfrak c$, and hence corresponds to a vertical annulus containing the singular axis at the centre.
Then the union of such weighted arcs is realised as a union of weighted arcs connecting $\partial D$ to $\mathfrak c$ intersecting only at $\mathfrak c$.

 \begin{figure}
 \psfrag{a}{$v_1=v_3$}
 \psfrag{b}{$v_4$}
 \psfrag{c}{$v_P$}
 \psfrag{d}{$v_6$}
 \psfrag{e}{$v_4$}
 \psfrag{f}{$v_2$}
 \psfrag{g}{$A_3$}
 \psfrag{h}{$A_1$}
 \psfrag{i}{$A_4$}
 \psfrag{j}{$P_0$}
 \psfrag{k}{$A_6$}
 \psfrag{l}{$A_4$}
 \psfrag{m}{$A_2$}
\resizebox{\textwidth}{!}{\includegraphics{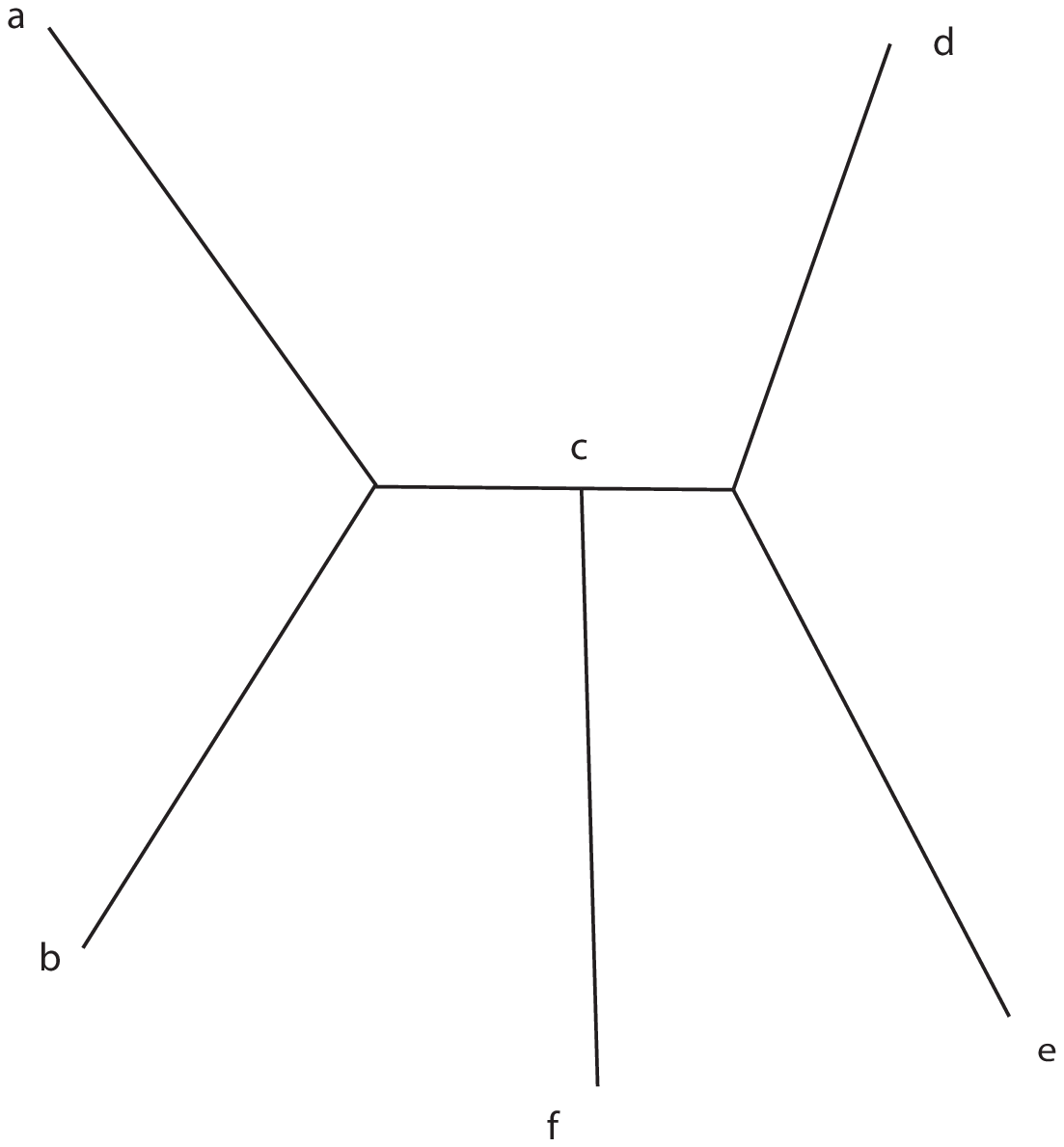} \hspace{2cm}
\includegraphics{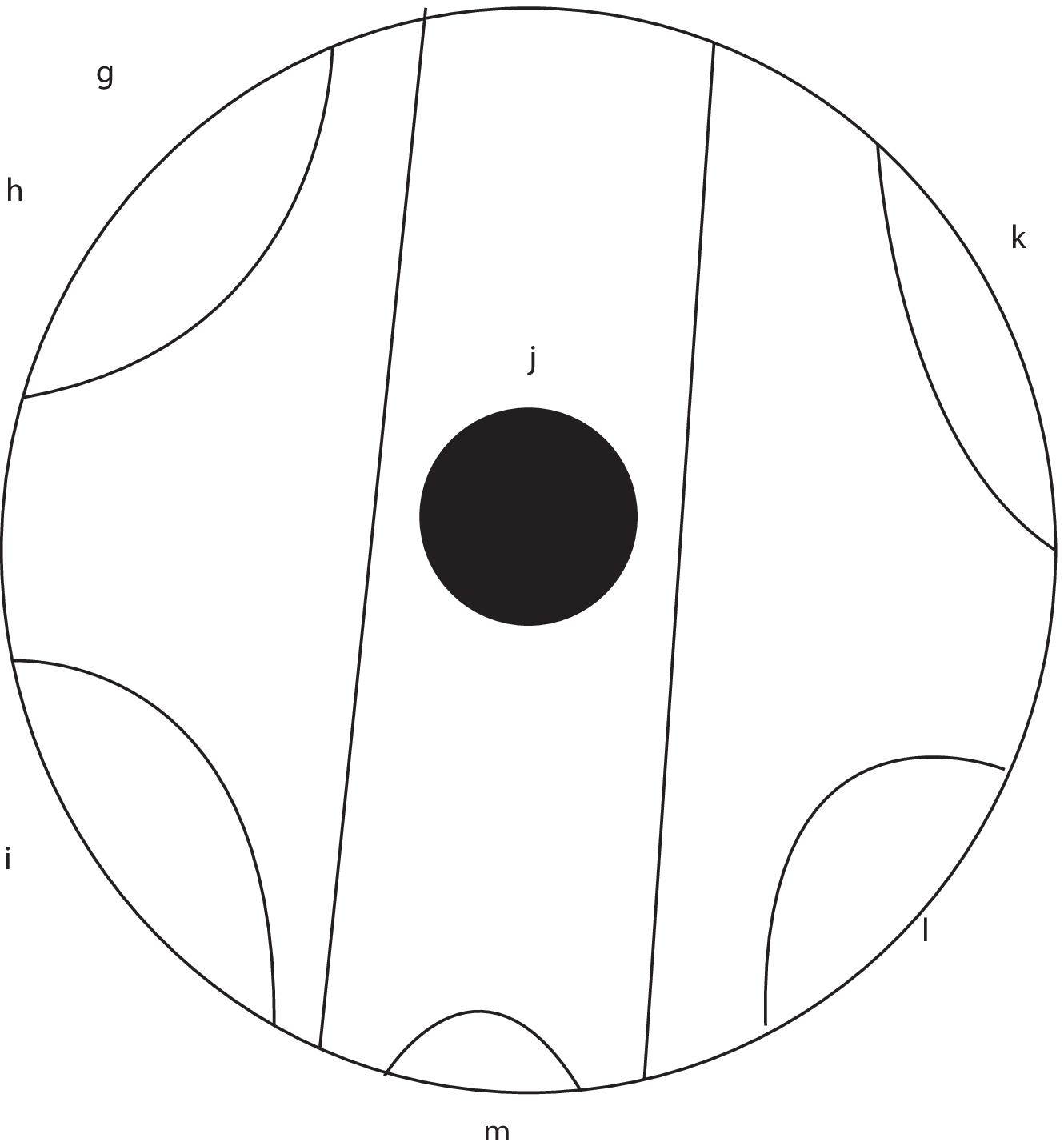}}
\caption{The tree $T'_V$ and its dual disc when $V$ is a thickened torus.}
 \label{tree}
\end{figure}
%


We perform shuffling around $V$ so that the order of the components of $\Fr V$ coincides with the order of the corresponding vertices of $\bar T_V$ on $D$.
Then we define a new annular codimension-1 lamination in $V$ to be the union of vertical annuli (possibly with singular axes) over the weighted union of arcs $\alpha$.
We note that if one vertex of $\bar T_V$  has more than one number, it does not matter how to perform shuffling among them.
Repeating this operation for every torus pair and thickened torus pair intersecting $\hat L$, we have a pared manifold $(M',P')$ obtained from $M$ by shuffling without modifying the paring locus $P$, and an annular codimension-1 measured lamination $L'$ (possibly containing annuli with singular axes) in $M'$ such that there is a $\pi_1(M')=\pi_1(M)$-equivariant isometry from the dual tree $T_{L'}$ of $L'$ to $T$.
\end{proof}

To derive \cref{rectified Thurston} from \cref{dual annular}, we need some facts relating three-dimensional hyperbolic geometry and two-dimensional one.
We first prove the following  rather elementary fact about degeneration of hyperbolic structures on surfaces.
This can be regarded as a two-dimensional analogue of the problematic second sentence of \cref{Thurston}.
\begin{lemma}
\label{surviving part}
Let $S$ be an orientable surface of hyperbolic type, and let $\{m_i\}$ be a sequence in $\teich(S)$.
Then, passing to a subsequence,  there is a possibly disconnected subsurface $S'$ of $S$ such that $\{m_i|S'\}$ converges in $\teich(S')$ and such that $S'$ is maximal up to isotopy among the subsurfaces with the same property.
\end{lemma}
\begin{proof}
Let $\alpha_i$  be a sequence of shortest pants decomposition of $(S, m_i)$, each of which is regarded as a collection of simple closed geodesics.
Passing to a subsequence, we can assume that there is a sub-collection $\beta_i$ of $\alpha_i$ which is constant, and no component of $\alpha_i \setminus \beta_i$ has a constant subsequence.
Since $\beta_i$ is constant, we denote it by $\beta$.
We take a union of $\beta$ and all pairs of pants in $S\setminus \alpha_i$ all of whose boundary components belong to $\beta$, and denote it by $\Sigma$.

Let $\Sigma_0$ be a component of $\Sigma$, and denote $\beta\cap \Sigma$ by $\beta_0$.
For a component $b$ of $\beta_0$, by passing to a subsequence, we can assume that either $\len_{m_i}(b) \to 0$ or $\len_{m_i}(b)$ converges to a positive constant.
In the first case, we cut $\Sigma_0$ along $b$.
In the second case, passing to a subsequence again, the twisting parameter (with respect to the Fenchel-Nielsen coordinate) around $b$ can be assumed to converge or goes to $\pm \infty$.
In the latter case, we cut $\Sigma_0$ along $b$.
Repeating this for all components of $\beta_0$, we get a subsurface $\Sigma_1$ of $\Sigma_0$, and $m_i|\Sigma_1$ converges in $\teich(\Sigma_1)$.
Taking the union of such $\Sigma_1$ for all components $\Sigma_0$ of $\Sigma$, we obtain $S'$. 
The maximality follows from our construction easily.
\end{proof}

Let  $(S,m)$ be a hyperbolic surface of finite area and $M$ be a complete hyperbolic 3-manifold.
A pleated surface $f \colon S \to M$ is a continuous map such that for each point $x$ on $S$, there is at least one segment $s$ containing $x$ at its interior such that $f|s$ is totally geodesic.
The set consisting of points where there is exactly one germ for such segments is called the pleating locus of $f$, and it constitutes a geodesic lamination on $(S, m)$.
The following is a consequence of \lq efficiency of pleated surfaces' proved in Theorems 3.3 and 3.4 in \cite{Th2} combined with \cref{surviving part}.

\begin{proposition}
\label{diverge/converge}
Let $(M,P)$ be a pared manifold, and let $\rho_i \in AH(M, P)$ be a sequence.
Then for every component $S$ of $\partial M \setminus P$, there is a possibly disconnected subsurface $S'$ such that $\{\rho_i|\pi_1(S')\}$ converges up to conjugacy, and $S'$ is maximal up to isotopy among the subsurfaces of $S$ with this property.
If $S'$ intersects essentially a component $T$ of $\partial M \setminus W$ for the characteristic submanifold $W$, then $S'$ contains the entire $T$ (up to isotopy).
\end{proposition}
\begin{proof}
Let $f_i \colon (S,m_i) \to \hyperbolic/\rho_i(\pi_1(M))$ be a pleated surface taking each frontier of $S$ to a cusp, which is efficient in the following sense:
\begin{itemize}
\item[(*)]
For any simple closed curve $\gamma$ on $S$, the translation length of $\rho_i(\gamma)$ is bounded if and only  $\{\len_{m_i}(\gamma)\}$ is bounded, where $\len_{m_i}$ denotes the geodesic length.
\end{itemize}
The existence of such a sequence of pleated surfaces was proved in Theorem 3.3 of Thurston \cite{Th2}.
(We need to take a covering associated with $\pi_1(S)$ to apply this result of Thurston.)
Now the first statement of our proposition follows from \cref{surviving part}.
Combining this with the first part of \cref{Thurston}, we obtain the second statement.
\end{proof}

Although this surface $S'$ is determined only up to isotopy, we always assume that no frontier component of $S'$ is isotopic into $P$, and any component $T$ of $\partial M \setminus W$ intersecting $S'$ is contained in $S'$.

We need to refine \cref{diverge/converge} for twisted $I$-pairs in the characteristic submanifold as follows.

\begin{corollary}
\label{twisted}
In the setting of \cref{diverge/converge}, suppose that there is a twisted $I$-pair $(W_0, W_0 \cap \partial M)$ in the characteristic submanifold of $(M,P)$, whose base surface we denote by $\Sigma_0$.
Let $S$ be a  component of $\partial M \setminus P$ containing $W_0 \cap \partial M$.
Then the subsurface $S'$ as obtained in \cref{diverge/converge} doubly covers a subsurface of $\Sigma_0$.
\end{corollary}
\begin{proof}
We can consider a non-orientable efficient pleated surface $g_i \colon \Sigma_0 \to \hyperbolic/\rho_i(\pi_1(M))$.
Consider a double cover $\tilde M$ of $M$ into which $W_0$ is lifted to a product $I$-bundle.
Then $g_i$ is lifted to an orientable pleated surface $\tilde g_i \colon (\tilde \Sigma_0, n_i) \to \hyperbolic/\rho_i(\pi_1(\tilde M))$, and both the  length with respect to $n_i$ and the translation length with respect to $\rho_i$ are invariant under the covering involution.
This implies \cref{twisted}.
\end{proof}

\begin{proof}[Proof of \cref{rectified Thurston}]
We have only to consider the case when $\{\rho_i\}$ diverges in $AH(M,P)$.
(Otherwise we can just take $X$ to be empty.)
Let $W$ be the characteristic submanifold of $(M,P)$.
For every  component $S$ of $\partial M\setminus P$, there is a subsurface $S'$ as was given in \cref{diverge/converge}.
Since any component of $\partial M \setminus W$ that intersects $S'$ must be entirely contained in $S'$, every frontier component of $S'$ is contained in either an $I$-pair or a solid torus pair or a thickened torus pair.

We only consider the frontier components contained in $I$-pairs.
Let $c$ be such a frontier component of $S'$ contained an $I$-pair $W_0$.
If $W_0$ is a product $I$-pair, then there is a simple closed curve $c'$ on $W_0 \cap \partial M$ such that $c\cup c'$ bounds an essential (vertical) annulus $A$ in $W_0$.
By our definition of $S'$, if $c'$ is contained in $S$, then it must also be a frontier component of $S'$.
Otherwise, let $\bar S$ be a component of $\partial M \setminus P$ containing $c'$.
Then we have a subsurface $\bar S'$ of $\bar S$ corresponding to $S'$ in \cref{diverge/converge}, and $c'$ is a frontier component of $\bar S'$.

If $W_0$ is a twisted $I$-pair, then by \cref{twisted}, either $c$ is homotopic to $c'$ as above and $c\cup c'$ bounds an annulus $A$, or $c$ doubly covers an orientation-reversing curve in the base surface.
In the latter case, $c$ bounds an essential one-sided M\"{o}bius band, which we also denote by $A$.

We remove from $M$ thin regular neighbourhoods of  all annuli and M\"{o}bius bands $A$ as above for all frontier components of $S'$ contained in $I$-pairs, and obtain  a 3-submanifold $M_0$ of $M$.
We let the union of $P$ and  the frontier components of $M_0$ in $M$, the latter of which are all annuli, be a paring locus $P_0$ of $M_0$.
By our construction, it is easy to check that $(M_0, P_0)$ is a pared manifold, and that $M \setminus M_0$ is contained in the union of $I$-pairs in $W$, and can be assumed to be vertical with respect to the fibration.
We let the projection of $M \setminus M_0$ to the base surface be $x$ in our statement, and hence $X$ coincides with $M \setminus M_0$.

Now restricting $\rho_i$ to $(M_0,P_0)$, we get a sequence $\{\rho_i'\}$ of $\pi_1(M_0)$.
If $\{\rho_i\}$ converges after passing to a subsequence, then by letting $\mathcal A$ be empty and setting $M'=M$, we are done.
Suppose that $\{\rho_i'\}$ does not have a convergent subsequence.
We note that the translation length of every curve on a component of $P_0$ with respect to $\rho_i$ is bounded, by our definition of $S'$ in \cref{diverge/converge}.
Therefore, we can apply \cref{dual annular}, and there are a pared manifold $(M_0', P_0)$ obtained from $(M_0, P_0)$ by shuffling and a codimension-1 annular measured lamination $L_0$ properly embedded in $(M_0', P_0)$ whose dual tree is equivariantly isometric to the Gromov limit of the action of $\pi_1(M_0)$ on rescaled $\hyperbolic$.
By performing the same shuffling that gives $M_0'$ from $M_0$, we obtain a pared manifold $(\hat M_0, P)$ homotopy equivalent of $(M,P)$.

By our definition of $M_0$, and by \cref{diverge/converge}, the lamination $L_0$ is disjoint from the $I$-pairs in the characteristic submanifold of $(M_0', P_0)$, and hence contained in the union of solid torus pairs and the thickened torus pairs.
This means that $L_0$ is a union of weighted incompressible annuli, possibly with singular axes.
We let $\mathcal A_0$ be the support of $L_0$.
Now, we cut $M_0'$ along $\mathcal A_0$, and obtain a 3-manifold $M_1$.
In the case when $\mathcal A_0$ contains an annulus with singular axis, we regard $M_0'$ as containing a regular neighbourhood of the singular axis on both sides.
By letting $P_1$ be the union of $\mathcal A_0 \cup (P_0 \cap M_1)$ when there are no singular axes, and when a component of $\mathcal A_0$ has a singular axis, we let the intersection of the regular neighbourhood of the singular axis with the boundary be contained in $P_1$, we get a pared manifold $(M_1, P_1)$.
We consider restrictions of $\rho_i$ to $\pi_1(M_1)$, and if they converge passing to a subsequence, we let $\mathcal A$ be $\mathcal A_0$, setting $(M', P')$ to be $(\hat M_0, P)$.
Otherwise, we repeat the same procedure for $(M_1, P_1)$ as for $(M_0, P_0)$, and obtain the next pared manifold $(M_2, P_2)$ contained in a pared manifold $(\hat M_1, P)$ obtained by shuffling from $(\hat M_0, P)$.
Since there are only finitely many disjoint incompressible annuli in $M$, this process terminates in finite steps, and there is $n_0$ such that $\rho_i|\pi_1(M_{n_0})$ converges.
Then we set $(\hat M_{n_0-1}, P)$ to be $(M',P')$, and  $\mathcal A=\cup_{j=0}^{n_1}\mathcal A_j$ regarded as contained in $(M',P')$.
Then together with $x$ which we defined above, also regarded as contained in $(M',P')$, we are done.
%
%
%
%
%
%
%
%
\end{proof}

\section{Implication of the bounded image theorem}
To explain the significance of \cref{Thurston} in the proof of Thurston's uniformisation theorem for Haken manifolds, we first review the overall structure of the proof.
Let $M$ be a closed atoroidal Haken manifold.
In general, we should consider a pared manifold $(M,P)$ and show the existence of a geometrically finite hyperbolic structure, but 
for simplicity, here we assume that $M$ is closed, and hence $P=\emptyset$.

By Waldhausen's result \cite{Wa}, a Haken manifold admits a \lq hierarchy', that is,  a finite sequence of compact 3-manifolds $M=M_0, M_1, \dots , M_n$ such that $M_{j+1}$ is obtained by cutting $M_j$ along a non-peripheral incompressible surface, and $M_n$ is a disjoint union of 3-balls.
The basic strategy of the proof of the uniformisation theorem is to construct a convex hyperbolic structure of $M_j$ from that of $M_{j+1}$.
Topologically $M_j$ is obtained from $M_{j+1}$ by pasting an incompressible subsurface on the boundary  of $M_{j+1}$ to another incompressible subsurface disjoint from the first one.
For a general step, we need to separate the boundary of $M_{j+1}$ into two parts: the part which is glued each other to get $M_j$, and the part which remains to be the boundary of $M_j$.
This makes the argument complicated.
Therefore, we here consider only the last step in which each boundary component $M_1$ is glued to another component to get $M_0$.

In the logical structure of the proof given by Thurston in the form of a chart on the page 206 of \cite{Th1}, \cref{Thurston} is meant to be used to prove the \lq bounded image theorem', which is a key step for showing that the \lq skinning map'  defined below has a fixed point in the case when $M_1$ is not an $I$-bundle over a surface.
The bounded image theorem in the original version is as below, restricted to the case which we are considering, \ie the case when the paring locus is empty.

To state the theorem, we need to define the skinning map.
Let $M$ be a 3-manifold admitting a convex compact hyperbolic metric.
Let $q \colon \teich(\partial M) \to QH_0(M)$ be the parametrisation of Marden.
For each element of $QH_0(M)$, by taking the covering associated with a component $S$ of $\partial M$, we get a quasi-Fuchsian representation, and by the Bers simultaneous uniformisation, we have a point in $\teich(S) \times \teich(\bar S)$.
We consider the projection to the second coordinate $\teich(\bar S)$, and taking the product by considering all the components of $\partial M$, we have a map $\sigma \colon \teich(\partial M) \to \teich(\bar \partial M)$, where the latter denotes the Teichm\"{u}ller space of $\partial M$ with its orientation reversed.

\begin{theorem}[The bounded image theorem]
\label{bounded image}\index{bounded image theorem}
Let $M$ be an atoroidal Haken manifold with incompressible boundary admitting a convex compact hyperbolic metric, and $\tau \colon \partial M \to \partial M$ be an orientation-reversing involution taking each component of $\partial M$ to another component.
Suppose that $N=M/\tau$ is also atoroidal.
Let $\sigma \colon \teich(\partial M) \to \teich(\bar \partial M)$ be the skinning map.
Then there is $n_0 \in \naturals$ depending only on the topological type of $M$ such that $(\tau_*\circ \sigma)^n \teich(\partial M)$ is precompact for every $n \geq n_0$.
\end{theorem}

There are three important books which explain  the proof of  the uniformisation theorem in detail: Morgan \cite{Mo}, Kapovich \cite{Kap} and Otal \cite{Otal-Haken}.
In the first two of them, the bounded image theorem is mentioned, but it is noted that its proof is unknown to the authors.
Instead, they gave a proof of  a weaker version of the theorem, which should be called the \lq bounded orbit theorem' as follows.
This theorem was first stated in Thurston's preprint \cite{ThO}.
To prove this, only the first part of \cref{Thurston}, which is valid, and some argument involving the \lq covering theorem' are necessary.
In the third book \cite{Otal-Haken}, the author invokes McMullen's results \cite{Mc1, Mc2}, which implies \cref{bounded orbit} without even using the first part of \cref{Thurston}.

\begin{theorem}[The bounded orbit theorem]
\label{bounded orbit}\index{bounded orbit theorem}
In the setting of \cref{bounded image}, let $g$ be any point in $\teich(\partial M)$.
Then the sequence $\{(\tau \circ \sigma)^n_*(g)\}$ has a convergent subsequence in $\teich(\partial M)$.
\end{theorem}

Quite recently, Lecuire and the present author \cite{LO} succeeded in proving the bounded image theorem in the original general form, in particular \cref{bounded image},  making use of modern technique of Kleinian group theory, but not using the second part of \cref{Thurston}.
In the process of proving this, it has become possible to imagine what was Thurston's plan to prove the bounded image theorem.
As far as we understand it, his proof essentially relies on the second part of \cref{Thurston}.
We here try to recover the \lq proof' which we speculate that he should have had in mind.

The \lq proof'  is by contradiction.
By the continuity of $\tau_* \circ \sigma$, we see that if $(\tau_* \circ \sigma)^n \teich(S)$ is precompact, then so is $(\sigma \circ \tau)^m$ for every $m \geq n$.
Therefore, what we have to exclude is the case when $(\tau_* \circ \sigma)^n_*\teich(S)$ is unbounded for every $n$.
If the image $(\tau_*\circ \sigma)^n \teich(S)$ is unbounded for every $n$, there must exist a sequence $\{m_i\}$ in $\teich(S)$ such that $\{(\tau_*\circ \sigma)^n m_i\}_i$ diverges in $\teich(S)$ for every $n$.
Then, passing to a subsequence, there are two possibilities to consider: the first is when $\{(q(\tau_* \circ \sigma)^n(m_i)\}_i$ diverges  in $AH(M)$; the second is when $\{q((\tau_*\circ \sigma)^n(m_i)\}_i$ converges in $AH(M)$.

For the first possibility, the second part of \cref{Thurston} is essential to reach a contradiction.
We assumed that $\{(q(\tau_*\circ \sigma)^n(m_i)\}_i$ diverges  in $AH(M)$ for every $n$.
Then, for a fixed $n$, provided that  the second part of \cref{Thurston} were valid, it would give us a collection $\mathcal A_n$ of incompressible annuli which separates $M$ into \lq converging part' and \lq diverging part' for every $n$.
It can be proved that this division into two parts is preserved under $\tau$, by the argument similar to the one employed in \S16, in particular Proposition 16.5 of \cite{Kap}.
Therefore for each annulus in $\mathcal A_n$, its boundary is mapped by $\tau$ to a simple closed curve which lies on the boundary of an annulus in $\mathcal A_{n+1}$. 
We see that by continuing to connect an annulus in $\mathcal A_n$ to that of $\mathcal A_{n+1}$ along a boundary starting from $n=1$,  we have either an incompressible torus or an incompressible Klein bottle in $N$, contradicting the assumption that $N$ is atoroidal.

Now, suppose that $\{q(\tau_* \circ \sigma)^n(m_i)\}$ converges in $AH(M)$ as $i \to \infty$ for every $n$.
Then the argument is quite similar to that of the bounded orbit theorem explained in \cite{Kap}, or the proof of the bounded image theorem for the special case when $M$ is acylindrical explained in \cite{Kent}.
We can consider the limit representation $\phi_\infty \colon \pi_1(M) \to \pslc$ of $\{q(\sigma \circ \tau)^n(m_i)\}$ as $i \to \infty$ for each $n$.
Then either $\phi_\infty(\pi_1(M))$ has a parabolic element  or $\hyperbolic/\phi_\infty(\pi_1(M))$ has a totally degenerate end, \ie, an end with a neighbourhood of a form $S \times (0,\infty)$ for a closed surface $S$ which is entirely contained in the convex core.
In the latter case, there is  a sequence of pleated surfaces $f_k \colon S \to \hyperbolic/\phi_\infty(\pi_1(M))$ tending to the end as $k \to \infty$.
These facts are proved by Bonahon \cite{Bo} in more general setting, but for the present case, there is an explanation by Thurston himself  in \S9 of his lecture notes \cite{ThL}.

In the first case when $\phi_\infty(\pi_1(M))$ has a parabolic element, there must be a simple closed curve $c$ on $\partial M$ representing a parabolic element.
(This follows from the fact that otherwise the limit $\rho_\infty$ is a strong limit.)
This will force $\tau(c)$ to be homotopic in $M$ to a another simple closed curve $c'$ on $\partial M$ homotopically distinct from $\tau(c)$, which implies that $\tau(c)$ and $c'$ bound an incompressible annulus in $M$.
We note here that we cannot prove this by considering the limit group $\phi_\infty(\pi_1(M))$ since $\hyperbolic/\phi_\infty(\pi_1(M))$ may not be homeomorphic to the interior of $M$ although they are homotopy equivalent.
Therefore we need to work in $\hyperbolic/q(\tau_*\circ \sigma)^n(m_i)$ for very large $i$.
The same argument can be found in \S17.2.3 in \cite{Kap}.
Once we have such an annulus, we repeat the same argument as the one which we used for the case when $\{q(\tau_* \circ \sigma)^n(m_i)\}$ diverges, and reach a contradiction to the assumption that $N$ is atoroidal.

We next consider the latter case when $\hyperbolic/\phi_\infty(\pi_1(M))$ has  a totally degenerate end.
Then we see that there must be one corresponding to a boundary component $S$ of $M$, for under the assumption that there is no parabolic element in $\phi_\infty(\pi_1(M))$, the convergence must be strong.
By the covering theorem, which first appeared in \S9 of \cite{ThL} (a generalised version can be found in Canary \cite{Ca}), this implies that $M$ must have a boundary component which is homotopic to $\tau(S)$, but is distinct from $\tau(S)$.
 This is impossible since $M$ is not a product $I$-bundle over a closed surface.
 
Thus we have reached a contradiction in either case, and this would give a proof of \cref{bounded image} if the second part of \cref{Thurston} were valid.
On the other hand, a weaker version as we showed in \cref{rectified Thurston} is not sufficient for the proof of \cref{bounded image} since to construct a torus or a Klein bottle to reach a contradiction, we need annuli embedded in $M$, not those embedded in a homotopy equivalent manifold $M'$.
This observation also highlights why a proof of the bounded image theorem given in \cite{LO} without relying on the second part of \cref{Thurston} needed much sophisticated technique of Kleinian group theory. 


\begin{thebibliography}{10}

\bibitem{AC}
{\sc Anderson, J.~W., and Canary, R.~D.}
\newblock {Algebraic limits of Kleinian groups which rearrange the pages of a
  book}.
\newblock {\em Inventiones mathematicae 126}, 2 (00 1996), 205 -- 214.

\bibitem{Ber}
{\sc Bers, L.}
\newblock {Simultaneous uniformization}.
\newblock {\em Bulletin of the American Mathematical Society 66}, 2 (00 1960),
  94 -- 97.

\bibitem{Bers2}
{\sc Bers, L.}
\newblock Spaces of {K}leinian groups.
\newblock In {\em Several {C}omplex {V}ariables, {I} ({P}roc. {C}onf., {U}niv.
  of {M}aryland, {C}ollege {P}ark, {M}d., 1970)\/} (1970), Springer, Berlin,
  pp.~9--34.

\bibitem{Be}
{\sc Bestvina, M.}
\newblock Degenerations of the hyperbolic space.
\newblock {\em Duke Math. J. 56}, 1 (1988), 143--161.

\bibitem{Bo}
{\sc Bonahon, F.}
\newblock Bouts des vari\'{e}t\'{e}s hyperboliques de dimension {$3$}.
\newblock {\em Ann. of Math. (2) 124}, 1 (1986), 71--158.

\bibitem{BBCL}
{\sc Brock, J., Bromberg, K., Canary, R., and Lecuire, C.}
\newblock {Convergence and divergence of Kleinian surface groups}.
\newblock {\em Journal of Topology 8}, 3 (00 2015), 811 -- 841.

\bibitem{BBCM}
{\sc Brock, J.~F., Bromberg, K.~W., Canary, R.~D., and Minsky, Y.~N.}
\newblock {Windows, cores and skinning maps}.
\newblock {\em Annales scientifiques de l'{\'E}cole normale sup{\'e}rieure 53},
  1 (2020), 173--216.

\bibitem{Ca}
{\sc Canary, R.~D.}
\newblock A covering theorem for hyperbolic {$3$}-manifolds and its
  applications.
\newblock {\em Topology 35}, 3 (1996), 751--778.

\bibitem{CM}
{\sc Canary, R.~D., and McCullough, D.}
\newblock {Homotopy equivalences of 3-manifolds and deformation theory of
  Kleinian groups}.
\newblock {\em Memoirs of the American Mathematical Society 172}, 812 (00
  2004), xii -- 218.

\bibitem{CS}
{\sc Culler, M., and Shalen, P.~B.}
\newblock {Varieties of group representations and splittings of 3-manifolds}.
\newblock {\em Annals of Mathematics 117}, 1 (00 1983), 109 -- 146.

\bibitem{JS}
{\sc Jaco, W.~H., and Shalen, P.~B.}
\newblock Seifert fibered spaces in {$3$}-manifolds.
\newblock {\em Mem. Amer. Math. Soc. 21}, 220 (1979), viii+192.

\bibitem{Jo}
{\sc Johannson, K.}
\newblock {\em Homotopy equivalences of {$3$}-manifolds with boundaries},
  vol.~761 of {\em Lecture Notes in Mathematics}.
\newblock Springer, Berlin, 1979.

\bibitem{Kap}
{\sc Kapovich, M.}
\newblock {\em Hyperbolic manifolds and discrete groups}.
\newblock Modern Birkh{\"a}user Classics. Birkh{\"a}user Boston, Inc., Boston,
  MA, 2009.
\newblock Reprint of the 2001 edition.

\bibitem{Kent}
{\sc Kent, R.~P.}
\newblock {Skinning maps}.
\newblock {\em Duke Mathematical Journal 151}, 2 (02 2010), 279 -- 336.

\bibitem{Kra}
{\sc Kra, I.}
\newblock On spaces of {K}leinian groups.
\newblock {\em Comment. Math. Helv. 47\/} (1972), 53--69.

\bibitem{LO}
{\sc Lecuire, C., and Ohshika, K.}
\newblock Thurston's bounded image theorem.
\newblock {\em to appear in Geom. Top.\/}.

\bibitem{Mar}
{\sc Marden, A.}
\newblock {The geometry of finitely generated kleinian groups}.
\newblock {\em Annals of Mathematics 99}, 3 (00 1974), 383 -- 462.

\bibitem{Mas}
{\sc Maskit, B.}
\newblock Self-maps on {K}leinian groups.
\newblock {\em Amer. J. Math. 93\/} (1971), 840--856.

\bibitem{Masbook}
{\sc Maskit, B.}
\newblock {\em Kleinian groups}, vol.~287 of {\em Grundlehren der
  mathematischen Wissenschaften [Fundamental Principles of Mathematical
  Sciences]}.
\newblock Springer-Verlag, Berlin, 1988.

\bibitem{Mc1}
{\sc McMullen, C.}
\newblock Amenability, {P}oincar\'{e} series and quasiconformal maps.
\newblock {\em Invent. Math. 97}, 1 (1989), 95--127.

\bibitem{Mc2}
{\sc McMullen, C.}
\newblock Iteration on {T}eichm\"{u}ller space.
\newblock {\em Invent. Math. 99}, 2 (1990), 425--454.

\bibitem{Mo}
{\sc Morgan, J.~W.}
\newblock On {T}hurston's uniformization theorem for three-dimensional
  manifolds.
\newblock In {\em The {S}mith conjecture ({N}ew {Y}ork, 1979)}, vol.~112 of
  {\em Pure Appl. Math.} Academic Press, Orlando, FL, 1984, pp.~37--125.

\bibitem{M-O}
{\sc Morgan, J.~W., and Otal, J.-P.}
\newblock Relative growth rates of closed geodesics on a surface under varying
  hyperbolic structures.
\newblock {\em Comment. Math. Helv. 68}, 2 (1993), 171--208.

\bibitem{MS1}
{\sc Morgan, J.~W., and Shalen, P.~B.}
\newblock Valuations, trees, and degenerations of hyperbolic structures. {I}.
\newblock {\em Ann. of Math. (2) 120}, 3 (1984), 401--476.

\bibitem{MS2}
{\sc Morgan, J.~W., and Shalen, P.~B.}
\newblock Degenerations of hyperbolic structures. {II}. {M}easured laminations
  in {$3$}-manifolds.
\newblock {\em Ann. of Math. (2) 127}, 2 (1988), 403--456.

\bibitem{MS3}
{\sc Morgan, J.~W., and Shalen, P.~B.}
\newblock Degenerations of hyperbolic structures. {III}. {A}ctions of
  {$3$}-manifold groups on trees and {T}hurston's compactness theorem.
\newblock {\em Ann. of Math. (2) 127}, 3 (1988), 457--519.

\bibitem{OhT}
{\sc Ohshika, K.}
\newblock {Divergence, exotic convergence and self-bumping in quasi-Fuchsian
  spaces}.
\newblock {\em Annales de la Facult{\'e} des sciences de Toulouse :
  Math{\'e}matiques 29}, 4 (00 2020), 805 -- 895.

\bibitem{Otal-Haken}
{\sc Otal, J.-P.}
\newblock Thurston's hyperbolization of {H}aken manifolds.
\newblock In {\em Surveys in differential geometry, {V}ol. {III} ({C}ambridge,
  {MA}, 1996)}. Int. Press, Boston, MA, 1998, pp.~77--194.

\bibitem{Pa}
{\sc Paulin, F.}
\newblock Topologie de {G}romov {\'e}quivariante, structures hyperboliques et
  arbres r{\'e}els.
\newblock {\em Invent. Math. 94}, 1 (1988), 53--80.

\bibitem{Sk}
{\sc Skora, R.~K.}
\newblock Splittings of surfaces.
\newblock {\em J. Amer. Math. Soc. 9}, 2 (1996), 605--616.

\bibitem{Su}
{\sc Sullivan, D.}
\newblock {Quasiconformal homeomorphisms and dynamics. II. Structural stability
  implies hyperbolicity for Kleinian groups}.
\newblock {\em Acta Mathematica 155}, 3-4 (00 1985), 243 -- 260.

\bibitem{ThL}
{\sc Thurston, W.}
\newblock {\em The Geometry and Topology of Three-manifolds: Lecture Notes from
  Princeton University 1978-80}.
\newblock Mathematical Sciences Research Institute, notes taken by Kerckhoff,
  S. and Floyd, W.J., 1978-1980.

\bibitem{ThC}
{\sc Thurston, W., Farb, B., Gabai, D., and Kerckhoff, S.}
\newblock {\em Collected Works of William P. Thurston with Commentary vol.2}.
\newblock Collected works series. American Mathematical Society, 2022.

\bibitem{ThO}
{\sc Thurston, W.~P.}
\newblock Hyperbolic structures on 3-manifolds: Overall logic.
\newblock In {\em Collected Works of William P. Thurston with Commentary
  vol.2}. American Mathematical Society.

\bibitem{Th1}
{\sc Thurston, W.~P.}
\newblock Hyperbolic structures on {$3$}-manifolds. {I}. {D}eformation of
  acylindrical manifolds.
\newblock {\em Ann. of Math. (2) 124}, 2 (1986), 203--246.

\bibitem{Th2}
{\sc Thurston, W.~P.}
\newblock Hyperbolic structures on 3-manifolds {II}: Surface groups and
  3-manifolds which fiber over the circle.
\newblock {\em https://arxiv.org/abs/math/9801045\/} (1998).

\bibitem{Th3}
{\sc Thurston, W.~P.}
\newblock {Hyperbolic Structures on 3-manifolds, III: Deformations of
  3-manifolds with incompressible boundary}.
\newblock {\em https://arxiv.org/abs/math/9801058\/} (1998).

\bibitem{Wa}
{\sc Waldhausen, F.}
\newblock {On irreducible 3-manifolds which are sufficiently large}.
\newblock {\em Annals of Mathematics 87}, 1 (1968), 56 -- 88.

\end{thebibliography}
\end{document}